\documentclass[preprint,12pt]{elsarticle}

\usepackage{amssymb}
\usepackage{amsthm}

\newcommand{\sysn}{\left\{\begin{array}{rcl}}
\newcommand{\sysk}{\end{array}\right.}

\newtheorem{theorem}{Theorem}[section]

\theoremstyle{example}

\theoremstyle{definition}
\newtheorem{definition}[theorem]{Definition}

\newtheorem{corollary}[theorem]{Corollary}

\journal{...}

\begin{document}

\begin{frontmatter}



\title{Group structures of a function spaces with the set-open topology}


\author{Alexander V. Osipov}

\ead{OAB@list.ru}


\address{Krasovskii Institute of Mathematics and Mechanics, Ural Federal
 University,\\ Ural State University of Economics, Yekaterinburg, Russia}

\begin{abstract}

In this paper, we find at the properties of the family $\lambda$
which imply that the space $C(X,\mathbb{R}^{\alpha})$ --- the set
of all continuous mappings of the space $X$ to the space
$\mathbb{R}^{\alpha}$ with the $\lambda$-open topology is a
semitopological group (paratopological group, topological group,
topological vector space and other algebraic structures) under the
usual operations of addition and multiplication (and
multiplication by scalars). For example, if $X=[0,\omega_1)$ and
$\lambda$ is a family of $C$-compact subsets of $X$, then
$C_{\lambda}(X,\mathbb{R}^{\omega})$ is a semitopological group
(locally convex topological vector space, topological algebra),
but $C_{\lambda}(X,\mathbb{R}^{\omega_1})$ is not semitopological
group.


\end{abstract}


\end{frontmatter}



\section{Introduction}

In articles \cite{os1, os2}, we investigated the
topological-algebraic properties of $C_{\lambda}(X)$ where
$C_{\lambda}(X)$ the space of all real-valued continuous functions
on $X$ with set-open topology. Recall that a subset $A$ of a space
$X$ is called
 {\it $C$-compact subset $X$} (or $\mathbb R$-compact) if, for any real-valued function~$f$ continuous on $X$,
 the set $f(A)$ is compact in~${\mathbb{R}}$. Given a family $\lambda$ of non-empty subsets of $X$, let
$\lambda(C)=\{A\in \lambda$ :  for every $C$-compact subset $B$ of
the space $X$ with $B\subset A$, the set $[B,U]$ is open in
$C_{\lambda}(X)$ for any open set $U$ of the space $\mathbb{R}
\}$.

\begin{theorem}(Theorem 3.3. in \cite{os2})\label{th1}
For a space $X$, the following statements are equivalent.

\begin{enumerate}

\item  $C_{\lambda}(X)=C_{\lambda, u}(X).$

\item  $C_{\lambda}(X)$ is a topological group.

\item  $C_{\lambda}(X)$ is a topological vector space.

\item $C_{\lambda}(X)$ is a locally convex topological vector
space.

\item  $\lambda$ is a family of\, $C$-compact sets and
$\lambda=\lambda(C)$.

\end{enumerate}

\end{theorem}

In this paper, we find at the properties of the family $\lambda$
which imply that the space $C_{\lambda}(X,\mathbb{R}^{\alpha})$
--- the set of all continuous mappings of the space $X$ to the
space $\mathbb{R}^{\alpha}$ with the $\lambda$-open topology is a
semitopological group (paratopological group, topological group,
topological vector space and other algebraic structures) under the
usual operations of addition and multiplication (and
multiplication by scalars).

\medskip

\section{Main definitions and notation}

\medskip

 Let $X$ be a Tychonoff space and $\alpha$ be
a cardinal number. We shall denote by $C(X,\mathbb{R}^{\alpha})$
the set of all continuous mappings of the space $X$ to the space
$\mathbb{R}^{\alpha}$. We say that a subset $B$ of a space $X$ is
$C_{\alpha}$-compact in $X$ if for every continuous function
$f:X\mapsto \mathbb{R}^{\alpha}$, $f(B)$ is a compact subset of
$\mathbb{R}^{\alpha}$ (\cite{Gar1}, \cite{Gar2}). This concept
generalizes the notion of $\alpha$-pseudocompactness introduced by
J.F. Kennison (\cite{ken}): a space $X$ is $\alpha$-pseudocompact
if $X$ is $C_{\alpha}$-compact in itself. If $\alpha\leq\omega_0$,
then we say $C$-compact instead of $C_{\alpha}$-compact and
$\alpha$-pseudocompactness agrees with pseudocompactness.

 A family
$\lambda$ of $C_{\alpha}$-compact subsets of $X$ is said to be
 closed under (hereditary with respect to) $C_{\alpha}$-compact subsets if it satisfies
 the following condition: whenever $A\in \lambda$ and $B$ is
 a $C_{\alpha}$-compact (in $X$) subset of $A$, then $B\in \lambda$ also.

For a Hausdorff space $X$, a family $\lambda$ of subsets of $X$
and a uniform space $(Y,\mathcal{U})$ we shall denote by
$\mathcal{U}|\lambda$ the uniformity on $C(X,Y)$ generated by the
base consisting of all finite intersections of the sets of the
form

$\hat{V}|A=\{(f,g): (f(x),g(x))\in V$ for every $x\in A\}$, where
$V\in \mathcal{U}$, $A\in \lambda$.

The uniformity $\mathcal{U}|\lambda$ will be called the uniformity
of uniform convergence on family $\lambda$ induced by
$\mathcal{U}$.

 Recall that all sets of the form

 $[F,U]:=\{f\in C(X,Y):\ f(F)\subseteq U\}$, where $F\in\lambda$ and $U$ is an open subset
 of $Y$, form a subbase of the set-open ($\lambda$-open) topology on
 $C(X,Y)$.

We use the following notations for various topological spaces
 on the set $C(X,\mathbb{R}^{\alpha})$:

 $C_{\lambda,u}(X,\mathbb{R}^{\alpha})$ for the topology induced by
 the uniformity $\mathcal{U}|\lambda$,

 $C_{\lambda}(X,\mathbb{R}^{\alpha})$ for the $\lambda$-open topology.

\medskip

Given a family $\lambda$ of non-empty subsets of $X$, let
$\lambda(C_{\alpha})=\{A\in \lambda$ :  for every
$C_{\alpha}$-compact subset $B$ of the space $X$ with $B\subset
A$, the set $[B,U]$ is open in
$C_{\lambda}(X,\mathbb{R}^{\alpha})$ for any open set $U$ of the
space $\mathbb{R}^{\alpha} \}$.

 Let $\lambda_m$ denote the maximal with respect to inclusion
 family, provided that $C_{\lambda_m}(X,\mathbb{R}^{\alpha})=C_{\lambda}(X,\mathbb{R}^{\alpha})$. Note that a family
 $\lambda_m$ is unique for each family $\lambda$.

 A subset $A$ of $X$ is said to be a {\it $Y$-zero-set}
provided $A=f^{-1}(Z)$ for some zero-set $Z$ of $Y$ and $f\in
C(X,Y)$. For example, if $Y$ is the real numbers space
$\mathbb{R}$ then any zero-set subset of $X$ is a $\mathbb
R$-zero-set of $X$.

\begin{definition}(for $\pi$-network see in \cite{os3}) A family $\lambda$ of subsets of $X$, hereditary with respect to
 the $\mathbb{R}^{\alpha}$-zero-subsets of $X$ will be called a {\it
saturation} family.
 \end{definition}







\begin{definition}(Definition 2.5 in \cite{bhk}) A family $\lambda$ of subsets of $X$ is called a functional refinement if for every
$A\in \lambda$, every finite sequence $U_1,..., U_n$ of open
subsets of $Y$, and every $f\in C(X,Y)$ such that $A\subseteq
\bigcup\limits_{i=1}^{n} f^{-1}(U_i)$, there exists a finite
sequence $A_1,..., A_m$ of members of $\lambda$ which refines
$f^{-1}(U_1),...,f^{-1}(U_n)$ and whose union contains $A$.
\end{definition}

$\bullet$ A group $G$ with a topology $\tau$ is a semitopological
(paratopological, respectively) group if the multiplication is
separately continuous (jointly  continuous, respectively).

$\bullet$ If $G$ is a semitopological and the inverse operation
$x\rightarrow x^{-1}$ is continuous, then $G$ is said to be a
quasitopological group.

Clearly, that if $G$ is a paratopological group and the inverse
operation $x\rightarrow x^{-1}$ is continuous, then $G$ is a
topological group.

$\bullet$ A topological algebra over a topological field
$\mathbb{K}$ is a topological vector space together with a
continuous bilinear multiplication.

 In \cite{burb}, N.Bourbaki noticed that if $Y$ is a topological ring then $C(X,Y)$ is a ring under the usual  operations of
pointwise addition and pointwise multiplication.

\section{Main results}

\begin{theorem}\label{th1} For a Tychonoff space $X$ and a cardinal number $\alpha$, the following statements are
equivalent.

\begin{enumerate}

\item  $C_{\lambda}(X,\mathbb{R}^{\alpha})$ is a semitopological
group.

\item  $C_{\lambda}(X,\mathbb{R}^{\alpha})$ is a paratopological
group.

\item  $C_{\lambda}(X,\mathbb{R}^{\alpha})$ is a quasitopological
group.

\item  $C_{\lambda}(X,\mathbb{R}^{\alpha})$ is a topological
group.

\item  $C_{\lambda}(X,\mathbb{R}^{\alpha})$ is a topological
vector space.

\item $C_{\lambda}(X,\mathbb{R}^{\alpha})$ is a locally convex
topological vector space.

\item $C_{\lambda}(X,\mathbb{R}^{\alpha})$ is a topological ring.

\item $C_{\lambda}(X,\mathbb{R}^{\alpha})$ is a topological
algebra.

\item  $\lambda$ is a family of\, $C_{\alpha}$-compact sets and
$\lambda=\lambda(C_{\alpha})$.

\item $\lambda_m$ is a saturation family of\, $C_{\alpha}$-compact
subsets of $X$.

\item
$C_{\lambda}(X,\mathbb{R}^{\alpha})=C_{\lambda,u}(X,\mathbb{R}^{\alpha})$.

\item $\lambda_m$ is a functional refinement family of\,
$C_{\alpha}$-compact subsets of $X$.

\end{enumerate}

\end{theorem}

\begin{proof} $(1)\Rightarrow(9)$. Let
$C_{\lambda}(X,\mathbb{R}^{\alpha})$ be a semitopological group.
Suppose that there exists $A\in \lambda$ such that $A$ is not a
$C_{\alpha}$-compact subset of $X$. Then there is $f\in
C_{\lambda}(X,\mathbb{R}^{\alpha})$ such that $f(A)$ is not a
compact subset of $\mathbb{R}^{\alpha}$. By Theorem 1.2 in
\cite{Gar1}, we can assume that $f(A)$ is not a closed subset of
$\mathbb{R}^{\alpha}$. Let us consider a point $a\in
\overline{f(A)}\setminus f(A)$ and the subbasic open set
$O(f):=[A, \mathbb{R}^{\alpha}\setminus \{a\}]$ which contains the
point $f$. Since $C_{\lambda}(X,\mathbb{R}^{\alpha})$ is a
semitopological group, there is a basis neighborhood  $[B,W]$ of a
point $h\equiv{\bf 0}$ such that $f+[B,W]\subseteq O(f)$. Choose
$x_0\in A$ such that $f(x_0)\in (a+W)$. Let $g\equiv a-f(x_0)$ be
a constant function. It is clear that $g\in
C_{\lambda}(X,\mathbb{R}^{\alpha})$ and $g\in [B,W]$, but
$(f+g)\notin O(f)$, because $f(x_0)+g(x_0)=a\notin
\mathbb{R}^{\alpha}\setminus \{a\}$. This contradicts our
assumption that $f+[B,W]\subseteq O(f)$. It follows that $A$ is a
$C_{\alpha}$-compact subset of $X$.

Suppose that $A\in \lambda$, $B\subset A$ and $B$ is a
$C_{\alpha}$-compact subset of $X$. We claim that $[B,U]$ is an
open set in the space $C_{\lambda}(X,\mathbb{R}^{\alpha})$ for
each open set $U$ in $\mathbb{R}^{\alpha}$. Let $f\in [B,U]$ and
$\mathcal{U}$ a uniformity on the space $\mathbb{R}^{\alpha}$. By
Lemma 8.2.5 in \cite{eng}, there exists a $V\in \mathcal{U}$ such
that $B(f(B),V) \subset U$ where $B(f(B),V):=\bigcup\limits_{z\in
f(B)} \{y \in \mathbb{R}^{\alpha}, (z,y)\in V\}$. The set $W=f+[A,
Int B(0,V)]$ is the open set in
$C_{\lambda}(X,\mathbb{R}^{\alpha})$. It remains to prove that
$W\subset [B, U]$. For $g\in W$ and $x\in B$ we have
$g(x)=f(x)+h(x)$ where $h\in [A, Int B(0,V)]$. It follows that
$g(x)\in U$ and $W\subset [B,U]$.

Recall that a space $Y$  be called  cub-space if for any $x\in
Y\times Y$ there are a continuous map $f$ from $Y\times Y$ to $Y$
and a point $y\in Y$ such that $f^{-1}(y)=x$ (see \cite{os3}).
Since $\mathbb{R}^{\alpha}\times\mathbb{R}^{\alpha}$ is
homeomorphic to $\mathbb{R}^{\alpha}$ for $\alpha\geq \aleph_0$ we
have that $\mathbb{R}^{\alpha}$ is cub-space for $\alpha\geq
\omega_0$.

One can see that if $Y$ is a Tychonoff space with countable
pseudocharacter contains a nontrivial path then $Y$ is a cub-
space. Really, if $Y$ is a Tychonoff space with countable
pseudocharacter then each point in $Y\times Y$ is a zero-set.  For
a point $(x,y)\in Y\times Y$ there is a continuous function
$f:Y\times Y \mapsto \mathbb{I}=[0,1]$ such that
$f^{-1}(0)=(x,y)$. Since $Y$ contains a nontrivial path,
$\mathbb{I}$ is homeomorphic to a subspace of $Y$. It follows that
$\mathbb{R}^{\alpha}$ is cub-space for $\alpha\in\omega_0$.

$(9)\Rightarrow(11)$. Let $\lambda$ is a family of\,
$C_{\alpha}$-compact sets and $\lambda=\lambda(C_{\alpha})$.
Without loss of generality we can assume that if $B$ is
$C_{\alpha}$-compact sets  and $B\subset A$ for some $A\in
\lambda$ then $B\in \lambda$.

By Propositions 2.2 and 2.3 in \cite{os3}, the  family $\lambda$
is hereditary with respect to the
$\mathbb{R}^{\alpha}$-zero-subsets of $X$ (i.e.  any nonempty
$A\bigcap B\in \lambda$ where $A\in \lambda$ and $B$ is a
$\mathbb{R}^{\alpha}$-zero-set of $X$). By Theorem 3.3 in
\cite{os3}, the topology on $C(X,\mathbb{R}^{\alpha})$ induced by
the uniformity $\mathcal{U}|\lambda$  of uniform convergence on
the family $\lambda$ coincides with the $\lambda$-open topology on
$C(X,\mathbb{R}^{\alpha})$. It follows that
$C_{\lambda}(X,\mathbb{R}^{\alpha})=C_{\lambda,u}(X,\mathbb{R}^{\alpha})$.

$(11)\Rightarrow(6)$. Now for each $A\in \lambda$ and $\beta\in
\alpha$, define the pseudo-seminorm $p_{A,\beta}$ on
$C(X,\mathbb{R}^{\alpha})$ by $p_{A,\beta}(f)=\min \{1, sup
\{|pr_{\beta}\circ f(x)|: x\in A\}\}$ where $pr_{\beta}:
\mathbb{R}^{\alpha}=\prod\limits_{\gamma<\alpha}
\mathbb{R}_{\gamma}\mapsto \mathbb{R}_{\beta}$ is the $\beta$-th
projection from $\mathbb{R}^{\alpha}$ onto $\mathbb{R}_{\beta}$.

Also for each $A\in \lambda$ and $\epsilon>0$, let
$V_{A,\beta,\epsilon}=\{f\in C(X,\mathbb{R}^{\alpha}) :
p_{A,\beta}(f)< \epsilon \}$.

Let $\mathcal{V}=\{V_{A,\beta_1,\epsilon}\bigcap...\bigcap
V_{A,\beta_s,\epsilon} : A\in \lambda, s\in \omega, \epsilon>0
\}$.

It can be easily shown that for each $f\in
C(X,\mathbb{R}^{\alpha})$, $f+\mathcal{V}:=\{f+V : V\in
\mathcal{V} \}$ form a neighborhood base at $f$. We say that this
topology is generated by the collection of pseudo-seminorms
$\{p_{A,\beta} : A\in \lambda, \beta\in \alpha\}$. Note that if we
choose $\epsilon\in (0,1)$, then for each $f\in
C(X,\mathbb{R}^{\alpha})$ and $A\in \lambda$, we have
$f+V\subseteq <f,A,\epsilon>$ for some $V\in \mathcal{V}$ and
$<f,A,\frac{\epsilon}{2}>\subseteq f+V$. This shows that the
topology of uniform convergence on $\lambda$ is the same as the
topology generated by the collection of pseudo-seminorms
$\{p_{A,\beta} : A\in \lambda, \beta\in \alpha\}$. We see from
this point of view that $C_{\lambda,u}(X,\mathbb{R}^{\alpha})$ is
a topological group with respect to addition. By
$C_{\lambda}(X,\mathbb{R}^{\alpha})=C_{\lambda,u}(X,\mathbb{R}^{\alpha})$
and the fact that $C_{\lambda}(X,\mathbb{R}^{\alpha})$ is a
topological (semitopological) group, we have that $\lambda$ is a
family of\, $C_{\alpha}$-compact subsets of $X$ (see the
implication $(1)\Rightarrow(9)$). It follows that the topology is
generated by the collection of seminorms $\{p_{A,\beta} : A\in
\lambda, \beta\in \alpha\}$. Consequently,
$C_{\lambda}(X,\mathbb{R}^{\alpha})$ is a locally convex
topological vector space.

$(6)\Rightarrow(5)$. It is immediate.

$(11)\Rightarrow(8)$. By the implication $(11)\Rightarrow(6)$,
$C_{\lambda}(X,\mathbb{R}^{\alpha})$ is a locally convex
topological vector space. It remains to prove the continuity of
the operation of multiplication. Let $W=[A,V]$ be a base
neighborhood of the point ${\bf 0}$ where $A\in \lambda$ and $V$
is an open set of $\mathbb{R}^{\alpha}$. Since
$\mathbb{R}^{\alpha}$ is a topological algebra, there is an open
set $V^1$ of $\mathbb{R}^{\alpha}$ such that $V^1*V^1\subseteq V$.
Let $W^1=[A,V^1]$. Then $W^1*W^1\subseteq W$. Really,
$W^1*W^1=\{f*g : f\in W^1, g\in W^1\}=\{f*g : f(A)\subseteq V^1$
and $g(A)\subseteq V^1\}$. Clearly, that $f(x)*g(x)\in V^1*V^1$
for each $x\in A$. It follows that $(f*g)(A)\subseteq V$ and
$W^1*W^1\subseteq W$.
 We prove that if $W=[A,V]$ be a base neighborhood of the
point ${\bf 0}$ and $f\in C(X,\mathbb{R}^{\alpha})$ then there is
an open set $V^1\ni 0$ such that $f(A)*V^1\subseteq V$ and
$V^1*f(A)\subseteq V$. Indeed, $g=f*h$ and $g^1=h^1*f$ where
$h,h^1\in W^1$. Then $g(x)=f(x)*h(x)\in f(A)*V^1$ and
$g^1(x)=h^1(x)*f(x)\in V^1*f(A)$ for each $x\in A$. Note that
$g(A)\subseteq V$ and $g^1(A)\subseteq V$.

By definitions of algebraic structures we have the next
implications:

$(8)\Rightarrow(7)\Rightarrow(4)$ and
$(8)\Rightarrow(5)\Rightarrow(4)\Rightarrow(3)\Rightarrow(2)\Rightarrow(1)$.

$(9)\Rightarrow(10)$. Let $\lambda$ is a family of\,
$C_{\alpha}$-compact sets and $\lambda=\lambda(C_{\alpha})$.
Without loss of generality we can assume that if $B$ is
$C_{\alpha}$-compact sets  and $B\subset A$ for some $A\in
\lambda$ then $B\in \lambda$.

Note that for any finite collection $A_1,..., A_k\in \lambda$ and
an open set $W$ of $\mathbb{R}^{\alpha}$ the set
$[A_1\bigcup...\bigcup A_k, W]$ is an open set of
$C_{\lambda}(X,\mathbb{R}^{\alpha})$. Hence, without loss of
generality we can assume that $\lambda$ is closed under finite
unions.

Let $\mu:=\{ B : B$ is a $C_{\alpha}$-compact subset of $X$ and
$B\subseteq \overline{A}$ for $A\in \lambda\}$. We prove that
$\lambda_m=\mu$. Note that if $A\in \lambda$, then
$\overline{A}\in \lambda_m$. Really, $[A,W]=[\overline{A},W]$ for
any open set $W$ of $\mathbb{R}^{\alpha}$ and $C_{\alpha}$-compact
subset $A$ of $X$.

1. $\mu\subseteq \lambda_m$. Let $B$ be a $C_{\alpha}$-compact
subset of $X$ and $B\subseteq \overline{A}$ for some $A\in
\lambda$. Consider a set $[B, V]$ for an open set $V$ of
$\mathbb{R}^{\alpha}$. Let $f\in [B,V]$. Since $f(B)$ is a compact
set there is a zero-set $Z$ of $\mathbb{R}^{\alpha}$ such that
$f(B)\subseteq Z\subseteq V$. Consider a zero-set
$f^{-1}(Z)\bigcap \overline{A}$. Since $A$ a $C_{\alpha}$-compact
subsets of $X$, $D=f^{-1}(Z)\bigcap A\neq\emptyset$, and, by
condition (9), $[D,V]$ is an open set in
$C_{\lambda}(X,\mathbb{R}^{\alpha})$. It follows that
$[\overline{D},V]$ is an open set in
$C_{\lambda}(X,\mathbb{R}^{\alpha})$, too. Note that $f\in
[\overline{D},V]\subseteq [B,V]$. Hence $[B,V]$ is an open set in
$C_{\lambda}(X,\mathbb{R}^{\alpha})$ for any open set $V$ in
$\mathbb{R}^{\alpha}$ and $B\in \lambda_m$.

2. $\lambda_m\subseteq \mu$. Suppose that $B\in \lambda_m$ and
$B\notin \mu$. Let $W$ be an open set in $\mathbb{R}^{\alpha}$
such that $W\neq\mathbb{R}^{\alpha}$. Then $[B,W]$ is an open set
in $C_{\lambda}(X,\mathbb{R}^{\alpha})$ and, hence, for $f\in
[B,W]$ there is a base neighborhood $\bigcap\limits_{i=1}^{k}
[A_i,W_i]$ of the point $f$ where $A_i=\overline{A_i}$ and $A_i\in
\mu$ for $i=\overline{1,k}$ such that $\bigcap\limits_{i=1}^{k}
[A_i,W_i]\subseteq [B,W]$. Since $B\notin \mu$, $B\setminus
\bigcup\limits_{i=1}^{k} A_i\neq\emptyset$. Let $h\in
C(X,\mathbb{R}^{\alpha})$ such that $h\upharpoonright
\bigcup\limits_{i=1}^{k} A_i=f$ and $h(z)\notin W$ for some $z\in
B\setminus \bigcup\limits_{i=1}^{k} A_i$. Then $h\in
\bigcap\limits_{i=1}^{k} [A_i,W_i]$, but $h\notin [B,W]$, a
contradiction.

$(10)\Rightarrow(11)$. By Lemma 3.7 and Corollary 3.8 in
\cite{bhk}, we have that
$C_{\lambda}(X,\mathbb{R}^{\alpha})=C_{\lambda_m}(X,\mathbb{R}^{\alpha})=C_{\lambda,u}(X,\mathbb{R}^{\alpha})=C_{\lambda_m,u}(X,\mathbb{R}^{\alpha})$.

$(10)\Rightarrow(12)$. Let $A\in \lambda_m$, $U_1,...,U_n$ be a
finite sequence of open subsets of $\mathbb{R}^{\alpha}$ and $f\in
C(X,\mathbb{R}^{\alpha})$ such that $A\subseteq
\bigcup\limits_{i=1}^{n} f^{-1}(U_i)$. For $y\in f(A)\bigcap U_i$
we choose a zero set $F(y)$ such that $y\in Int F(y)\subseteq
F(y)\subseteq U_i$. Then $\{Int F(y): y\in f(A)\bigcap U_i$ and
$i=1,...,n\}$ is an open cover of the compact set $f(A)$. There
exists a finite sequence $F(y_1),..., F(y_m)$ such that
$f(A)\subset \bigcup_{j=1}^{m} IntF(y_j)$. Then
$A_j=f^{-1}(F(y_j))$ is zero-set of $X$ for $j=1,...m$. Since
$\lambda_m$ is a saturation family, $B_j=A\bigcap A_j\in
\lambda_m$. Since $A\subset \bigcup\limits_{j=1}^{m} A_j$, then
$A=\bigcup\limits_{j=1}^{m} B_j$, and the finite sequence
$B_1,...,B_m$ refines $f^{-1}(U_1),..., f^{-1}(U_n)$.

$(12)\Rightarrow(11)$. By Theorem 2 in \cite{br1}, replacing an
admissible family by a functional refinement family.
\end{proof}

Note that a space $C_{\lambda}(X,\mathbb{R}^{\alpha})$ is a
Hausdorff if and only if $\lambda$ is a $\pi$-network of $X$.

\begin{corollary} For a Tychonoff space $X$ and a cardinal number $\alpha$, the following statements are
equivalent.

\begin{enumerate}

\item  $C_{\lambda}(X,\mathbb{R}^{\alpha})$ is a Hausdorff
semitopological group.

\item  $C_{\lambda}(X,\mathbb{R}^{\alpha})$ is a Hausdorff
paratopological group.

\item  $C_{\lambda}(X,\mathbb{R}^{\alpha})$ is a Hausdorff
quasitopological group.

\item  $C_{\lambda}(X,\mathbb{R}^{\alpha})$ is a Hausdorff
topological group.

\item  $C_{\lambda}(X,\mathbb{R}^{\alpha})$ is a Hausdorff
topological vector space.

\item $C_{\lambda}(X,\mathbb{R}^{\alpha})$ is a Hausdorff locally
convex topological vector space.

\item $C_{\lambda}(X,\mathbb{R}^{\alpha})$ is a Hausdorff
topological ring.

\item $C_{\lambda}(X,\mathbb{R}^{\alpha})$ is a Hausdorff
topological algebra.

\item  $\lambda$ is a $\pi$-network of $X$ consisting of\,
$C_{\alpha}$-compact sets and $\lambda=\lambda(C_{\alpha})$.

\item $\lambda_m$ is a $\pi$-network of $X$ and is a saturation
family of\, $C_{\alpha}$-compact sets.

\item $\lambda_m$ is a $\pi$-network of $X$ and is a functional
refinement family of\, $C_{\alpha}$-compact sets of $X$.
\end{enumerate}

\end{corollary}

\section{Example}

The following results was obtained in \cite{Gar1}

\begin{theorem}(Theorem 2.6 in \cite{Gar1})\label{ga} Let $\alpha$ be a cardinal
with $cf(\alpha)>\omega$. Then $[0,\alpha)$ is
$\gamma$-pseudocompact for all $\omega\leq \gamma<cf(\alpha)$ and
it is not $cf(\alpha)$-pseudocompact.
\end{theorem}

\begin{corollary}(Corollary 2.8 in \cite{Gar1})\label{cr28} If $\gamma$ and $\alpha$ are cardinals with
$\gamma<\alpha$, then $[0,\gamma^{+})$ is $\gamma$-pseudocompact
and is not $\alpha$-pseudocompact.
\end{corollary}

It then we have the following result

\begin{theorem} Let $\alpha$ be a cardinal
with $cf(\alpha)>\omega$, $\omega\leq \gamma<cf(\alpha)$ and
$\lambda$ be a family of $C_{\gamma}$-compact subsets of
$X=[0,\alpha)$. Then $C_{\lambda}(X,\mathbb{R}^{\gamma})$ is a
semitopological group (locally convex topological vector space,
topological algebra), but $C_{\lambda}(X,\mathbb{R}^{cf(\alpha)})$
is not semitopological group.
\end{theorem}

\begin{proof} By Theorem \ref{ga}, $X$ is $\gamma$-pseudocompact and
is not $cf(\alpha)$-pseudocompact. By Theorem \ref{th1},
$C_{\lambda}(X,\mathbb{R}^{\gamma})$ is a semitopological group
(locally convex topological vector space, topological algebra),
but $C_{\lambda}(X,\mathbb{R}^{cf(\alpha)})$ is not
semitopological group.

\end{proof}

For example, if $X=[0,\omega_1)$ and $\lambda$ is a family of
$C$-compact subsets of $X$, then
$C_{\lambda}(X,\mathbb{R}^{\omega})$ is a semitopological group
(locally convex topological vector space, topological algebra),
but $C_{\lambda}(X,\mathbb{R}^{\omega_1})$ is not semitopological
group.

\end{document}